\documentclass[preprint,12pt,authoryear]{elsarticle}

\usepackage{amssymb}
\usepackage{natbib}

\usepackage{amssymb,amscd}
\usepackage{amscd}
\usepackage{amsmath,latexsym,amssymb,amsthm}
\usepackage{hyperref}
\usepackage{tikz}

\theoremstyle{plain}
\newtheorem{defn}[subsection]{Definition}
\newtheorem{thm}[subsection]{Theorem}
\newtheorem{lem}[subsection]{Lemma}

\newtheorem{note}[subsection]{Note}

\newtheorem{rem}[subsection]{Remark}

\journal{Journal of Algebra}

\begin{document}

\begin{frontmatter}


\title{Hilbert-Samuel Multiplicity of a Bipartite Graph\tnoteref{}}
\author{Priya Das\fnref{label2}}
\ead{priya.math88@gmail.com}
\author{Himadri Mukherjee}
\ead{himadri.mukhujje@gmail.com}
\address{IISER Kolkata, Mohanpur Campus}
\fntext[label2]{The author thanks UGC for finantial support}




\begin{abstract}
 Let $G$ be a bipartite graph and $I(G)$ the toric ideal associated to the graph $G$. In this article
we calculate Hilbert-Samuel multiplicity of the graph $G$ for which the toric ideal $I(G)$ is generated by a quadratic
binomials and it satiesfies some conditions.
\end{abstract}


\end{frontmatter}


\section{Introduction}
Let $K$ be a field and $G$ be a finite simple connected graph with vertex set $V(G)=\{1,2,...,d\}$ and edge set
$E(G)=\{e_1,e_2,...,e_n\}$. Let $K[x]=K[x_1,x_2,...,x_n]$ denote the polynomial ring in $n$-variables over $K$
and $K[t]=K[t_1,t_2,...,t_d]$ denote the polynomial ring in $d$-variables over the field $K$. For each edge $e=(i,j)$
where $i,j\in$ $V(G)$ we write $t^e:=t_it_j$. Let $K[G]$ denote the subalgebra of $K[t]$ which is generated by 
$t^{e_1},t^{e_2},...,t^{e_n}$ over $K$. Define $f: K[x]\rightarrow K[G]$ the surjective homomorphism by $f(x_i)= t^{e_i}$,
$1 \leq i \leq n$. Let $I(G)$ be the kernel of $f$. This ideal is studied by many author (ref. \cite{oh-hb},\cite{strm},
\cite{schu},\cite{him},\cite{brown},\cite{hibi}). Notably B. Strumfels related the binomial ideals with the toric varieties in
(\cite{strm}). In (ref.\cite{vil}) the authors calculated the edge cone of the affine toric variety $V(I(G))$ for a bipartite 
graph $G$. They also found the faces of the edge cone in terms of the edge vectors, which is given by the column vectors 
of the adjacency matrix of the bipartite graph. T. Hibi and H. Ohsugi (\cite{oh-hb}, Theorem 1.2) gives the following
conditions for the toric ideal $I(G)$ of a graph to be generated by quadratic binomials:

\begin{thm} 
\label{cond}
Let $G$ be a finite connected graph having no loops and no multiple edges. Then, the toric ideal $I(G)$ of $G$ is
generated by quadratic binomials if and only if the following conditions are satiesfied:
\begin{flushleft}
(i) If $C$ is an even cycle of $G$ of length $\geqslant$ $6$, then either $C$ has an even-chord or $C$ has three odd-chords 
$e,e^{'},e^{''}$ such that $e$ and $e^{'}$ cross in $C$;\\
(ii) If $C_1$ and $C_2$ are minimal odd cycles having exactly one common vertex, then there exists an edge
$\{i,j\}\notin E(C_1) \cup E(C_2)$ with $i\in V(C_1)$ and $j\in  V(C_2)$;\\
(iii) If $C_1$ and $C_2$ are minimal odd cycles having no common vertex, then there exist at least two bridges between
$C_1$ and $C_2$.
\end{flushleft}
\end{thm}

Using the above conditions one can show that if $K[G]$ is Koszul then it is normal (see \cite{oh-hb}, corollary 1.3). 
Let $G$ be a bipartite graph that satiesfies the condition of the above theorem (\ref{cond}) and let $I(G)=\langle f_1,f_2,...,f_r\rangle$, where
$f_i$'s are the quadratic binomial generators. In the present article we discuss a standard monomial basis for the tangent cone of the variety
$V(I(G))$. We define the standard monomial basis using a set of degree two monomials, which we call special monomials. For each
cycle $(e_1,e_2,e_3,e_4)$ we choose either one of the monomials $x_1x_3$ or $x_2x_4$ from the binomial relation 
$x_1x_3-x_2x_4$ defined by this
primitive cycle. While doing so we make sure that the special monomials thus chosen are mutually coprime and also the following
conditions are met. If we denote $\mathcal{NS}_2$ by the set of special monomials and $\mathcal{S}_2$ be the set of monomials
$\{w_i : f_i= w_i-z_i(\neq 0)\in I(G)$ for some $z_i\in \mathcal{NS}_2 \}$, then for each $t\in \mathbb{N}$ we want 
$\prod_{i=1}^{t} z_i\neq \prod_{i=1}^{t} w_i$. If we meet these requirements then we define the standard monomials 
to be the monomials that are not divisible by any of these special monomials $z_i$. We prove in theorem
(\ref{thm3}) that the set of standard monomials thus defined of degree $k$ i.e. $\mathcal{S}_k$ gives a basis for the vector space
$\mathfrak{m}^k/ \mathfrak{m}^{k+1}$, where $\mathfrak{m}$ is a maximal ideal of the ring $R=K[x_1,x_2,...,x_n]/I(G)$. We write
down that dim $\mathfrak{m}^k/ \mathfrak{m}^{k+1}=\lvert \mathcal{S}_k \rvert$, in a simple formula involving $n$ and $k$. As a result we get a simple
formula for Hilbert-Samuel multiplicity of $V(I(G))\subset \mathbb{A}^n$ at the origin
\begin{thm} 
Let $\bold{e}$ be the Hilbert-Samuel multiplicity of the ring $R$ associated to the given graph $G$. Then 
$\bold{e}$ is given by
\begin{center}
$\bold{e}={(m-2)! \over (n-1)!}\sum_{r=0}^{p} (-1)^{r} {p \choose r}\{ \sum_{i=0}^{n-m+1} (-1)^{i}
{n \brack m+i-1}{m+i-2 \choose i}(2r)^{i}\}$.
\end{center}
\end{thm}

\noindent As a particular result of the above theorem we have
\begin{thm}
 If the graph $G$ has $n$ number of $4$-cycles and it has $2n+2$ vertices and $3n+1$ edges, then the formula
 for the Hilbert-Samuel multiplicity $\bold{e}$ is given by $\bold{e}=2^n$.
\end{thm}

\noindent Roots of standard monomial theory (in short SMT)
can be found in classical works of Hodge (see \cite{hod1}, \cite{hod2}). A few more modern treatment of the techniques 
can be found in {\cite{lak-raghu}. C. S. Seshadri and V. Lakshmibai used SMT techniques on classical invariant theory 
(Ref.\cite{connection}). An important application of SMT is the determination of the singular loci of Schubert varieties 
(Ref.\cite{schu}), the reader is advise to look at these reference for a detailed application of such techniques.\\ This article is arranged in the following manner, in section 2, we introduce some terminology,
construct standard monomials of our graph and calculate the number of standard monomials of particular degree.
In section 3, we prove that standard monomials give a basis and in section 4, we give our main results. A formula for the Hilbert-Samuel multiplicity of the variety associated to the bipartite graph.

\section{Construction of the standard monomial}

Let $G$ be a bipartite graph with $m$-vertices and $n$-edges and the toric ideal $I(G)$ associated to the graph $G$ is 
generated by quadratic binomials. Let $R=K[x_1,x_2,...,x_n]/I(G)$ be the ring associated to the graph $G$. Now take $S$ 
the set of all primitive $4$-cycles in the graph $G$ and $\lvert S \rvert=p$(say). We associate each edge $e_i$ in $G$
as a variable $x_i$. Now we define the following:

\begin{defn} The set of all non-standard monomial of degree $2$ in the polynomial ring $R$ is denoted by $\mathcal{NS}_2$ 
and is defined as, let a 4-cycle be $(e_1,e_2,e_3,e_4)$ then a non-standard monomial is either $x_1x_3$ or $x_2x_4$.
Let $NS_2$ be the cardinality of $\mathcal{NS}_2$. We write

\begin{center}
$\mathcal{NS}_2=\{z_i : z_i= x_{i_1}x_{i_2}, i=1,2,...,p\}$ and $NS_2=p$.
\end{center}
\end{defn}

\noindent There is a choice involved in the above definition where one can choose all the non-standard monomials so that they are
mutually coprime then we call the graph ``good''.

\begin{defn} Denote $\mathcal{M}_2$ the set of all monomials of degree 2. The set of all standard monomials of degree $2$ 
in the polynomial ring $R$ is denoted by $\mathcal{S}_2$ and 
is defined by
\begin{center}
$\mathcal{S}_2=\{m\in \mathcal{M}_2 : z_i\nmid m, \forall i=1,2,...,p\}$. 
\end{center}
\end{defn}

\begin{note}
\label{note}
$f_i=w_i-z_i$, be the quadratic binomial for each 4-cycle, where $w_i$ and $z_i$'s are standard and non-standard monomials
of degree 2 respectively. Hence $I(G)=\langle f_i :i=1,2,...,p\rangle$.
\end{note}

\noindent Now we give the following definition of our graph

\begin{defn} Let $G$ be a bipartite graph with $m$-vertices and $n$-edges for which the toric ideal $I(G)$ associated to 
the graph $G$ is generated by quadratic binomials and it satiesfies the condition $\prod_{i=1}^{t} w_i \neq \prod_{i=1}^{t}
z_i$, for some $t$, where $w_i$ and $z_i$'s are the corresponding standard and non-standard monomials of degree 2
in a 4-cycle respectively of the graph $G$. Henceforth, all graph will be ``good''.
\end{defn}

\begin{defn} Denote $\mathcal{M}_k$ as the set of all monomials of degree $k$ in $R$ and $M_k$ as the cardinality of 
$\mathcal{M}_k$.
\end{defn}

\begin{defn} The set of all non-standard monomials of degree $k$ in $R$ is denoted by $\mathcal{NS}_k$ and is defined by all
monomials of degree $k$ which divides at least one $z_i$'s and denote the cardinality of $\mathcal{NS}_k$ by $NS_k$. We write
\begin{center}
 $\mathcal{NS}_k=\{ m\in \mathcal{M}_k : z_i\mid m$, for at least one $z_i$'s$\}$.
\end{center}
\end{defn}

\begin{defn} The set of all standard monomials of degree $k$ in $R$ is denoted by $\mathcal{S}_k$ and is defined by all 
monomials of degree $k$ which does not divide $z_i$'s and denote the cardinality of $\mathcal{S}_k$ by $S_k$. We write
\begin{center}
$\mathcal{S}_k=\{ m\in \mathcal{M}_k : z_i\nmid m,\forall i\}$.
\end{center}
\end{defn}

\noindent Now we state the following lemma
\newpage
\begin{lem}
$M_k={n+k-1 \choose n-1}$.
\end{lem}

\begin{proof}
As there are $n$-variables and we want all monomials of degree $k$, so $M_k$ is the number of non-negetive 
integer solutions of the equations  $x_1+x_2+...+x_{n}=k$, which is equal to ${n+k-1 \choose n-1}$.
\end{proof}

\noindent Now we define the following notation:\\
Let us denote $Y_i^k:=\{m\in \mathcal{M}_k : z_i\mid m\}, i=1,2,...,p$, which is a monomial of degree $k$.

\begin{lem}
\label{lem1}
$\lvert Y_i^k \rvert=M_{k-2}$ and $\lvert Y_{i_1}^{k}\cap Y_{i_2}^{k}\cap...\cap Y_{i_l}^{k}\rvert= M_{k-2l}
,l=1,2,...,p$.
\end{lem}

\begin{proof}
From above we have $Y_i^k:=\{m\in \mathcal{M}_k : z_i\mid m\}$.
So for $m\in Y_{i}^k$ since $z_i|m$, then $m=z_im_1$, for $m_1\in \mathcal{M}_{k-2}$.
So, for each $m$ we get different $m_1$.
Therefore $\lvert Y_i^k \rvert=M_{k-2}$.\\
For the second part, $Y_{i_1}^{k}\cap Y_{i_2}^{k}\cap...\cap Y_{i_l}^{k}=\{ m\in \mathcal{M}_k: z_i\mid m,\forall i=1,2,...,l\}$.
Hence $m=z_1z_2...z_lm_2, m_2\in \mathcal{M}_{k-2l}$.
Therefore, $\lvert Y_{i_1}^{k}\cap Y_{i_2}^{k}\cap...\cap Y_{i_l}^{k}\rvert= M_{k-2l},l=1,2,...,p$.
 
\end{proof}

\noindent Now we calculate the number of non-standard monomials of degree $k$.

\begin{lem}
$NS_k=\sum_{r=1}^{p} (-1)^{r-1}{p \choose r} M_{k-2r}$.
\end{lem}

\begin{proof}
We use inclusion-exclusion principle to derive the identity\\
\begin{align*}
NS_k &= \lvert \cup_{i=1}^{p} Y_i^k \rvert\\
&= \sum_{i=1}^{p} \lvert Y_i^{k} \rvert -{p \choose 2} \lvert Y_{i_1}^{k}\cap Y_{i_2}^{k}\rvert +
{p \choose 3} \lvert Y_{i_1}^{k}\cap Y_{i_2}^{k}\cap Y_{i_3}^{k}\rvert -...
\\&\hspace*{.5cm}+
(-1)^{p-1} {p \choose p} \lvert Y_1^{k}\cap Y_2^{k}\cap...\cap Y_p^{k}\rvert\\
&= p M_{k-2}-{p \choose 2}M_{k-4}+ {p \choose 3} M_{k-6}-....+(-1)^{p-1}{p \choose p} M_{k-2p}\\
&= \sum_{r=1}^{p} (-1)^{r-1} {p \choose r} M_{k-2r}.
\end{align*}

where the third equality follows from lemma (\ref{lem1}).
\end{proof}

\noindent At the junctor we can find the number of standard monomials of degree $k$.

\begin{lem}
\label{la1}
$S_k=\sum_{r=0}^{p} (-1)^{r}{p \choose r} M_{k-2r}$.
\end{lem}

\begin{proof}
As $S_k=M_k\setminus NS_k$. Hence the lemma follows.
\end{proof}

\section{Linear Independence and Generation}

\begin{thm}
\label{thm1}
Let $G$ be a bipartite graph and $R=K[x_1,...,x_n]/I(G)$ be the ring. Let $\mathfrak{m}=\langle x_1,x_2,...,x_n \rangle$ 
be a maximal ideal of $R$, then the set of all standard monomials of degree $k$ i.e. $\mathcal{S}_k$ in the vector 
space $\mathfrak{m}^k/ \mathfrak{m}^{k+1}$ are linearly independent over $R/\mathfrak{m}\cong K$.
 
\end{thm}

\begin{proof}
Let $S_1=\{m_1,m_2,...,m_r\}$ be the set of all standard monomials of degree $k$, where $m_i\neq m_j$ for all
$i\neq j$. Let  $\sum a_{i}m_{i}=0$ in $\mathfrak{m}^k/ \mathfrak{m}^{k+1}$, where $a_i\in K$ equivalantly
$\sum a_{i}m_{i} \in \mathfrak{m}^{k+1}$. Let $\sum a_{i}m_{i}=g\in \mathfrak{m}^{k+1}$.
Therefore, $\sum a_{i}m_{i}-g=0$ in $R$. Equivalantly $\sum a_{i}m_{i}-g \in I(G)$. 
Now since $I(G)$ is a homogeneous ideal then each element of $I(G)$ is homogeneous.
Now since $g\in \mathfrak{m}^{k+1}$, this implies $g=0$. Therefore $\sum a_{i}m_{i}\in I(G)$.
Let $I(G)=\langle f_1,f_2,...,f_l\rangle$ for some $l$. Now the rest proof of the theorem follows from the following
lemma (\ref{lem2})
\end{proof}

\begin{lem}
\label{lem2}
With the same notation as above theorem, if $\sum a_{i}m_{i}\in I(G)$ then $a_i=0$ for all $i$.
\end{lem}

\begin{proof}Let $\sum a_{i}m_{i}=\sum b_{j}f_{j}$ and assume that at least one of $a_i\neq 0$, without loss of generality let $a_1\neq 0$ 
and take $a_1=1$.\\
Therefore,
\begin{eqnarray}
m_1+ \sum_{i\neq 1} a_{i}m_{i}=\sum b_{j}f_{j}=\sum b_{j}(w_j-z_j)\label{eqn1}
\end{eqnarray}
Let us also assume that the above representation is minimal in sense that if we write
$b_j = \sum_{i=1}^{l} \alpha_i \mu_{ij}$ where $\mu_{ij}$ are 
monomials and $\alpha_i \neq 0$. We are assuming that there are no $\beta_{ij}$ (not all zero) such that 
$\sum_{j } \sum_{i=1}^{l} \beta_{ij}m_{ij}f_j=0$.\\
Let 
\begin{eqnarray}
m_1=c_1 w_1 \label{eqn2}
\end{eqnarray}
where $c_1$ and $w_1$ are standard monomials. Therefore, $-c_1z_1$ can be cancelled with either of $c_jz_j$ or $c_jw_j$.
If $-c_1z_1=c_jz_j$ then either $z_j\mid c_1$ or
$z_j\mid z_1$. But $z_j\nmid c_1$ as by (\ref{eqn2}) $c_1$ is standard, so $z_j\mid z_1\Rightarrow z_j=z_1$, by our construction.
This implies $w_j=w_1\Rightarrow -c_1=c_j$. Therefore $c_jf_j=-c_1f_1\Rightarrow c_jf_j+ c_1f_1=0$, which is a contradiction
as minimal representation of (\ref{eqn1}) is zero. So $-c_1z_1\neq c_j z_j$. Therefore $-c_1z_1=c_jw_j$. Now again $-c_j z_j=c_rs_r$,
where $s_r=z_r$ or $w_r$. Continuing in this way we get $c_ps_p=0$, where $s_p=w_p$ or $z_p$. This implies $c_p=0$ as $w_p$ and
$z_p$ can't be zero. Therefore we will get $c_1=0\Rightarrow m_1=0$, which is a contradiction as $m_1$ can't be zero. Therefore 
the only possibility for $\sum a_{i}m_{i}=\sum b_{j}f_j$ is $a_i=0$, for all $i$.
\end{proof}

\noindent Let us define the following notation:

\begin{defn} For $m$ be a standard monomial, define $\lvert m \rvert=\#\{z_i : z_i\mid m\}$.
\end{defn}

\begin{thm}
\label{thm2}
Let $G$ be a bipartite graph and $R$ be the ring then the set of standard monomials of degree $k$ i.e. $\mathcal{S}_k$
generates the space $\mathfrak{m}^k/\mathfrak{m}^{k+1}$ over $K$.
 
\end{thm}

\begin{proof} To prove the generation we use an induction argument on $\lvert m \rvert$. Note that, 
if $\lvert m \rvert=0$ then there does not exist $z_i$ such that $z_i\mid m$, then $m$ is standard. Assume that standard monomials generate all non-standard monomials for  $\lvert m \rvert \leq n-1$.
Now we prove for all $m$ with $\lvert m \rvert=n$, then there must exist $z_1$ such that $z_1\mid m$ as $n\geq 1$.
Therefore, $m={m}_1{z_1}=m_1 w_1 + I(G)$. Now if  $\lvert m_1 w_1 \rvert \textless n$ then by induction 
we are done. Else, $\lvert m_1 w_1 \rvert = n$. Now if any $z_i$ that divides $m_1 w_1$, then we replace $m_1 w_1$ by $\frac{m_1 w_1}{z_2} w_2=m_2$ (say), if $\lvert m_2 \rvert \textless n$ then
we are done again. Else we continue, but this process cannot go for ever, it has to stop somewhere, if for some $m_i$ there is
$k$ such that $m_i=m_{i+k}$ then we have the following
\begin{center}
  $m_i=m_i \frac{w_{i+1}}{z_{i+1}}\frac{w_{i+2}}{z_{i+2}}...\frac{w_{i+k}}{z_{i+k}}$
\end{center}

\noindent Now since $m_i\neq 0$ we have
\begin{center}
 $\frac{w_{i+1}}{z_{i+1}}\frac{w_{i+2}}{z_{i+2}}...\frac{w_{i+k}}{z_{i+k}}=1$\\
or, $\prod w_i=\prod z_i$
\end{center}

 \noindent Which is a contradiction as by our definition of graph. Hence the theorem follows.
\end{proof}
\noindent Now combining theorem (\ref{thm1}) and (\ref{thm2}) we get the following result

\begin{thm}
\label{thm3}
Let $G$ be the graph and $R$ be the ring. Then the set of all standard monomials of degree $k$
i.e. $\mathcal{S}_k$ forms a basis for the vector space $\mathfrak{m}^k/ \mathfrak{m}^{k+1}$ over the field $K$.
\end{thm}

\begin{thm} $dim_K$ $\mathfrak{m}^{k}/\mathfrak{m}^{k+1}=S_k$.
 
\end{thm}

\begin{proof}
 It follows from the above theorem.
\end{proof}

\section{Hilbert-Samuel multiplicity}

In this section, we calculate Hilbert-Samuel multiplicity of $V(I(G))\subset \mathbb{A}^n$ at the origin of a ring
$R=K[x_1,x_2,...,x_n]/I(G)\cong K[G]$. Now we know that the Krull-dimension of $K[G]$ is $m-1$, as there is $m$-vertices
(ref.\cite{oh-hb}). We need rising factorial power and Stirling number of first kind in this section. Let us recall few basic
concepts as detailed treatment can be found in ref.\cite{con}. \\
The rising factorial power is\\
$x^{\overline{k}}= x(x+1)...(x+k-1), k\in \mathbb{N}, k\geq 1$\\
And also we have the following\\
$x^{\overline{k}}=\sum_{l=1}^{k} {k \brack l} x^l$.

Let us recall few definitions and results about Hilbert-Samuel multiplicity. Let $R$ be a local Noetherian ring with maximal
ideal $\mathfrak{m}$ and let $M$ be a finitely generated module of dimension $d+1$. Let $\mathfrak{q}\subseteq \mathfrak{m}$ be
an ideal of $R$ such that $M/ \mathfrak{q}M$ has finite length. Then the Hilbert function $H_M(n)$ is defined by
\begin{center}
$H_M(n)=$length$(\mathfrak{q}^nM/\mathfrak{q}^{n+1}M)$.
\end{center}
$H_M(n)$ agrees with some polynomial function $P_\mathfrak{q}(M)(n)$ for large values of $n$, called Hilbert-Samuel polynomial
of $M$ with respect to $\mathfrak{q}$.

\begin{defn} The Hilbert-Samuel multiplicity of $\mathfrak{q}$ on $M$ is denoted by $e(\mathfrak{q},M)$ and is defined as $d!$
times the coefficient of the term of degree $d$ in the polynomial $P_\mathfrak{q}(M)$. 
\end{defn}
\begin{rem}
\label{rem1}
dim $M$= deg$P_\mathfrak{q}(M)+1$.
\end{rem}
\begin{lem} 
\label{lem3}
Let $M$ be a module of dimension $d+1$ and $\mathfrak{q}$ be an ideal of $R$ such that $M/ \mathfrak{q}M$ has
finite length. Then
\begin{center}
$e(\mathfrak{q},M)=d! \lim\limits_{n\to \infty} \frac{\rm length(M/q^n M)}{n^d}$.
\end{center}
\end{lem}

\begin{proof} Let $P_\mathfrak{q}(M)(n)=a_d n^d+a_{d-1} n^{d-1}+...+a_0$ be the polynomial of $M$. By above definition, the leading
term of this polynomial is $e(\mathfrak{q},M)/d!$. For $n \gg 1$, $P_{\mathfrak{q}}(M)(n)=\rm length (M/ \mathfrak{q}^n M)$. So we
get\\
\begin{eqnarray*}
d! \lim_{n \to \infty} \frac{\rm length (M/ \mathfrak{q}^n M)}{n^d}&=&d! \lim_{n \to \infty} \frac{P_\mathfrak{q}(M)(n)}{n^d}\\
&=& d! a_d\\
&=& e(\mathfrak{q},M).
\end{eqnarray*}

\end{proof}

\noindent In the special case where $M=R$ and $\mathfrak{q}=\mathfrak{m}$. We need to find the coefficient of $k^{dim R-1}$ in $\mathcal{S}_k$,
since $\mathcal{S}_k=\rm length(M/ \mathfrak{q}^n M)$.
\begin{lem}
 The coefficient $k^{m-2}$ in $S_{k}$ is
\begin{center}
${1 \over (n-1)!} \sum_{r=0}^{p} (-1)^{r} {p \choose r}\{ \sum_{i=0}^{n-m+1} (-1)^{i}
{n \brack m+i-1}{m+i-2 \choose i}(2r)^{i}\}$.

\end{center}

\end{lem}

\begin{proof} We have from (\ref{la1}) $S_k=\sum_{r=0}^{p} (-1)^{r}{p \choose r} M_{k-2r}$\\

Now 
\begin{eqnarray*}
M_{k-2r} &=& {n+k-2r-1 \choose n-1} \\
                   & =&{k+n-1-2r \choose n-1}\\
                    &=& {(k+n-1-2r)! \over (n-1)!(k-2r)!}\\
                    &=& {(u+1)(u+2)...(u+n-1) \over (n-1)!},\\
                    &=& {u^{\overline{n-1}} \over (n-1)!},            
\end{eqnarray*}
where $u=k-2r$.\\
Therefore, the coefficient of $k^{m-2}$ in $u^{\overline{n-1}}$ is

\begin{eqnarray*}
&=&\text{coefficient of }k^{m-2}\text{ in }\{ \sum_{l=1}^{n} {n \brack l} u^{l-1} \}\\
          &=& \text{coefficient of }k^{m-2}\text{ in }\{{n \brack m-1}(k-2r)^{m-2}+...+ {n \brack n}(k-2r)^{n-1}\}\\
          &=&{n \brack m-1}.1+(-1){n \brack m}{m-1 \choose 1}.2r+...
          \\& &\hspace{.5cm}+(-1)^{n-m+1}{n \brack n}{n-1 \choose n-m+1}(2r)^{n-m+1}\\
          &=& \sum_{i=0}^{n-m+1}(-1)^{i} {n \brack m+i-1}{m+i-2 \choose i}(2r)^{i}\\
\end{eqnarray*}

\noindent Therefore the coefficient of $k^{m-2}$ in $S_{k}$ is\\ \\
$={1 \over (n-1)!} \sum_{r=0}^{p} (-1)^{r} {p \choose r}\{ \sum_{i=0}^{n-m+1} (-1)^{i}
{n \brack m+i-1}{m+i-2 \choose i}(2r)^{i}\}$.

\end{proof}

\begin{thm} 
\label{thm4}
Let $\bold{e}$ be the Hilbert-Samuel multiplicity of the ring $R$ associated to the given graph $G$. Then 
$\bold{e}$ is given by
\begin{center}
$\bold{e}={(m-2)! \over (n-1)!}\sum_{r=0}^{p} (-1)^{r} {p \choose r}\{ \sum_{i=0}^{n-m+1} (-1)^{i}
{n \brack m+i-1}{m+i-2 \choose i}(2r)^{i}\}$.
\end{center}
\end{thm}

\begin{proof}As the dimension of $K[G]$ is $(m-1)$ therefore by remark (\ref{rem1}) we have $d=m-2$, therefore by lemma
(\ref{lem3}) we have
\begin{eqnarray*}
\bold{e}&=&(m-2)!\lim_{k \to \infty} {S_k \over k^{m-2}}\\
  &=&{(m-2)! \over (n-1)!}\sum_{r=0}^{p} (-1)^{r} {p \choose r}\{ \sum_{i=0}^{n-m+1} (-1)^{i}
{n \brack m+i-1}{m+i-2 \choose i}(2r)^{i}\},
\end{eqnarray*}
\end{proof}

\noindent An application of the above theorem we get the following lemma 

\begin{thm}
 If the graph $G$ has $n$ number of $4$-cycles and it has $2n+2$ vertices and $3n+1$ edges, then the formula for the Hilbert-
 Samuel multiplicity $\bold{e}$ is given by $\bold{e}=2^n$.

\end{thm}

\begin{proof} Here $m=2n+2$ and $n=3n+1$.\\
Therefore, from above we have
\begin{eqnarray*}
 \bold{e} &=&\frac{(2n)!}{(3n)!} \sum_{r=0}^{n} {n \choose r}\{ \sum_{i=0}^{n}{ 3n+1 \brack 2n+i+1 }{2n+i \choose i} (2r)^i \}\\
 &=& \frac{(2n)!}{(3n)!} \sum_{r=0}^{n} {n \choose r} \{(-1)^n {3n+1 \brack 3n+1} {3n \choose n} (2r)^n \}\\
 &=&\frac{(-1)^n 2^n}{n!} \sum_{r=0}^{n} (-1)^r {n \choose r} r^n\\
 &=&\frac{(-1)^n 2^n}{n!}(-1)^n n!\\
 &=&2^n.
\end{eqnarray*}

\end{proof}


\end{document}